\documentclass[12pt]{amsart}\usepackage{graphicx}
\usepackage{amsaddr}
\usepackage[T1]{fontenc}
\makeatletter\@namedef{subjclassname@2010}{\textup{2010} Mathematics Subject Classification}\makeatother

\newtheorem{thm}{Theorem}[section]
\theoremstyle{definition}
\newtheorem{defn}[thm]{Definition}
\newtheorem{rem}[thm]{Remark}

\newcommand{\N}{\mathbb{N}}
\newcommand{\R}{\mathbb{R}}
\newcommand{\ra}{\rightarrow}

\begin{document}
\title[A class of continua$\ldots$]{A class of continua that are not attractors of any IFS}
\author{Marcin Kulczycki}
\address[M. Kulczycki]{Faculty of Mathematics and Computer Science\\ Jagiellonian University\\
ul. \L o\-jasiewicza 6, 30-348 Krak\'ow, Poland}
\email{Marcin.Kulczycki@im.uj.edu.pl}
\author{Magdalena Nowak}
\address[M. Nowak]{Faculty of Mathematics and Computer Science\\ Jagiellonian University\\
ul.\L o\-jasiewicza 6, 30-348 Krak\'ow, Poland \\{\rm and}\\Institute of Mathematics\\Jan Kochanowski University\\ul. \'Swi\k{e}tokrzyska 15, 25-406 Kielce, Poland}
\email{magdalena.nowak805@gmail.com}

\begin{abstract}
This paper presents a sufficient condition for a continuum in $\R^n$ to be embeddable in $\R^n$ in such a way that its image is not an attractor of any iterated function system. An example of a continuum in $\R^2$ that is not an attractor of any weakly contracting iterated function system is also given.
\end{abstract}
\maketitle
\section{Introduction}

The notion of an iterated function system (abbrev. IFS), introduced by John Hutchinson in 1981 \cite{Hutch}, has proven to be a fertile field of research as well as a versatile and useful tool in lossy data compression (especially where image data is concerned). This paper is a study in one specific aspect of the theory - the possibility of encoding a particular set as an attractor of an IFS. We now recall some basic terminology.

Let $(X,d)$ be a complete metric space. A map $f:X\ra X$ is called a {\em contraction} if there exists a constant $\lambda\in(0,1)$ such that for every $x,y\in X$ we have $d(f(x),f(y))\leq\lambda d(x,y)$. A map $f:X\ra X$ is called a {\em weak contraction} if for every $x,y\in X$, $x\neq y$ we have $d(f(x),f(y))< d(x,y)$. A family $F=\{f_1,\dots,f_n\}$ of (weak) contractions $f_i:X\ra X$ is called a {\em (weakly contracting) iterated function system} (see \cite{B}). Given a compact $B\subset X$ define
$$F(B)=\bigcup_{i=1}^n f_i(B).$$
This transformation, acting on the space of nonempty compact subsets of $X$ with the Hausdorff metric, is called the {\em Barnsley-Hutchinson operator}.

It is shown in \cite{Hutch} that every IFS has a unique attractor. Analogous fact is also true for any weakly contracting IFS. M. Hata proved in \cite{Hata} that if the attractor of some IFS is connected, then it is also locally connected. M. J. Sanders showed in \cite{MS} that every arc of finite length is an attractor for some IFS. Additionally, he has proven that if $a$ is an endpoint of some arc $A\subset\R^n$ which has the properties:
\begin{enumerate}
\item for all $x, y \in A\setminus\{a\}$ the length of the subarc of $A$ with endpoints $x$ and $y$ is finite,
\item for every $x\in A\setminus\{a\}$ the length of the subarc of $A$ with endpoints $x$ and $a$ is infinite,
\end{enumerate}
then $A$ is not an attractor of any IFS acting on $\R^n$. One example of such arc is the harmonic spiral \cite{MS2}. The example of M. Kwieci\'nski from \cite{Kw} may also be easily modified to satisfy these assumptions.

It is elementary to check that every continuum in $\R$ is an attractor of some IFS. Moreover, any embedding of such continuum in $\R$ is still an attractor of some IFS. In dimesion two and higher, however, the situation becomes more complex. Our results provide a sufficient condition for a continuum to be embeddable in $\R^n$ so that its image is not an attractor of any IFS.

\section{Main Results}

\begin{defn}
Let $(X,d)$ be a metric space, $A\subset X$, $x,y\in A$, and $\varepsilon>0$. Consider all the sequences $x_1,\ldots,x_k$ such that $k\in\N$, $x_1=x$, $x_k=y$, $x_i\in A$, $d(x_i,x_{i+1})<\varepsilon$. Denote by $\tilde{d}(x,y,A,\varepsilon)$ the infimum of the sums $\sum_{i=1}^{k-1}d(x_i,x_{i+1})$ for these sequences. Define $\tilde{d}(x,y,A)=\lim_{\varepsilon\searrow 0}\tilde{d}(x,y,A,\varepsilon)$. This limit may be infinite.

It is elementary that if $A\subset B$ then $\tilde{d}(x,y,A)\geq\tilde{d}(x,y,B)$.
\end{defn}

\begin{thm}\label{main}
Let $n\geq 2$. Let $C\subset\R^n$ be a continuum. Assume that there exists an $(n-1)$-dimensional hyperspace $B\subset\R^n$ such that $B\cap C=\{p\}$ and $C\setminus\{p\}$ is connected. Assume additionally that for every $x,y\in C\setminus\{p\}$ there exists $U_{xy}$ which is a neighbourhood of $p$ such that $\tilde{d}(x,y,C\setminus U_{xy})<+\infty$. Then there exists an embedding $h:C\ra\R^n$ such that $h(C)$ is not an attractor of any IFS.
\end{thm}

\begin{proof}
By applying an affine transformation we may assume without loss of generality that $B=\{0\}\times\R^{n-1}$, $p=(0,\ldots,0)$, and $C\subset[0,1]\times[-1,1]^{n-1}$. Next define $h_1,h_2:\R^n\ra\R^n$ as
$$
h_1(x_1,\ldots,x_n)=(x_1,\frac{x_1}{100}x_2,\ldots,\frac{x_1}{100}x_n)
$$
$$
h_2(x_1,\ldots,x_n)=(x_1,\sqrt{x_1}\sin x_1^{-1}+x_2,x_3,\ldots,x_n)
$$
Then define the embedding $h:C\ra\R^n$ as the composition $h_2\circ h_1$.

Speaking colloquially, $h_1$ transforms $C$ into a sharp needle, while $h_2$ bends that needle to fit into a thickened-up graph of the function $\sqrt{x}\sin x^{-1}$. As a result of the second transformation the needle becomes, speaking imprecisely, of infinite length. Figure \ref{needle} illustrates the process for $n=2$.

\begin{figure}\centering\includegraphics[width=300px]{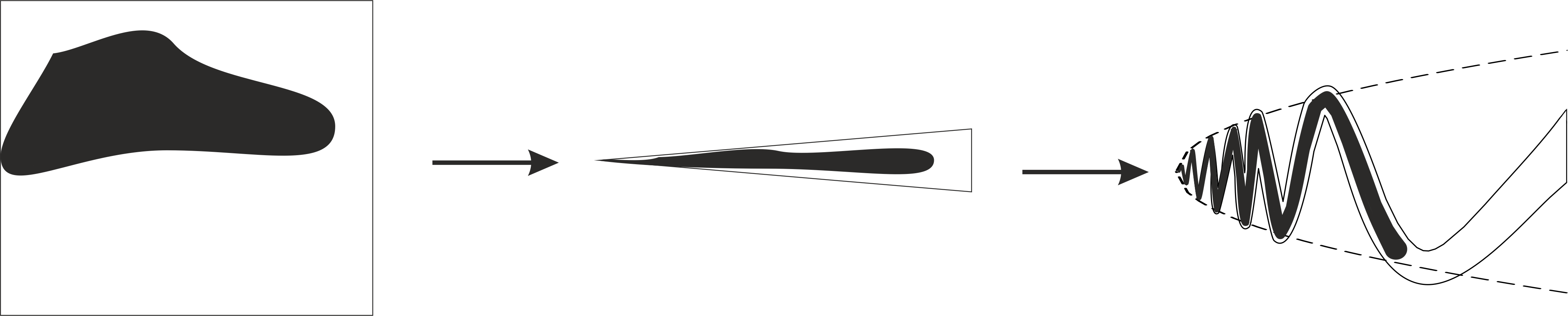}\caption{The map $h$ for $n=2$}\label{needle}\end{figure}

The map $h_1$ does not increase distance, and therefore for every $x,y\in h_1(C\setminus\{p\})$ there exists $U^1_{xy}$ which is a neighbourhood of $h_1(p)$ such that $\tilde{d}(x,y,h_1(C)\setminus U^1_{xy})<+\infty$.

Note that, outside of any neighbourhood $U$ of $h_1(p)$, the expansivity constant of $h_2|_{h_1(C)\setminus U}$ is bounded from above. This implies that for every $x,y\in h(C\setminus\{p\})$ there exists $U^2_{xy}$ which is a neighbourhood of $h(p)$ such that $\tilde{d}(x,y,h(C)\setminus U^2_{xy})<+\infty$.

Consider now a contraction $f:h(C)\ra h(C)$ with a Lipschitz constant $\lambda<1$. We would like to prove that if $h(p)\in f(h(C))$ then $f(h(C))=\{h(p)\}$. To this end, conjecture that $f$ is not constant and $h(p)\in f(h(C))$.

Assume first that $f(h(p))=h(p)$. Fix any $x\in h(C)$ such that $f(x)\neq h(p)$. Note that the sequence $x,f(x),f^2(x),\ldots$ is convergent to $h(p)$. Also note that by the assumptions $\tilde{d}(x,f(x),h(C))$ is finite and additionally $\tilde{d}(f^i(x),f^{i+1}(x),h(C))\leq\lambda^i\tilde{d}(x,f(x),h(C))$. But this would imply that $\tilde{d}(x,h(p),h(C))$ is also finite, while it is not, since it can be seen from the definition of $h_2$ that $\tilde{d}(x,h(p),h([0,1]\times[-1,1]^{n-1})$ is infinite.

If, on the other hand, $f(h(p))\neq h(p)$ then there exist $x\in h(C)$ such that $f(x)=h(p)$ and $y\in h(C)\setminus\{h(p)\}$ such that $f(y)\neq h(p)$. Then $\tilde{d}(x,y,h(C))$ would be finite and $\tilde{d}(f(x),f(y),h(C))$ would be infinite, which contradicts the contractivity of $f$, completing the proof that if $f$ takes value $h(p)$ on at least one argument then it has to be constant.

Consequently, if $F$ is the Barnsley-Hutchinson operator for some IFS and $F(h(C))\subset h(C)$, then $F(h(C))$ may comprise of $\{h(p)\}$ and possibly also finitely many other closed sets not containing $h(p)$. But then $F(h(C))\neq h(C)$, proving that $h(C)$ is not an attractor of $F$.
\end{proof}

\begin{rem}
The assumptions of Theorem \ref{main} are technical and may seem very restricitve. Its assertion, however, is true not only for the continua that satisfy them directly, but also for the continua that may be embedded in $\R^n$ in such a way that their image satisfies them. This significantly widens the class of sets the theorem is useful for. For example, if any two points in the continuum $A\subset\R^n$ can be connected in $A$ by a path of finite length, then it can be easily seen that the wedge sum of $A$ and $[0,1]$ may be embedded in $\R^{n+1}$ so that the assumptions of Theorem \ref{main} are satisfied.
\end{rem}

After the result of Hata \cite{Hata} it has been an open problem whether every locally connected continuum in $\R^n$ is an attractor of some IFS. The example of Kwieci\'nski \cite{Kw} provided a negative answer, but the same question for weakly contracting IFS's remains, to our best knowledge, open. We shall now give an example of a subcontinuum of $\R^2$ that is not an attractor of any weakly contracting IFS.

\begin{defn}
We now switch to a polar coordinate system $(r,\theta)$ on $\R^2$. Put $p_0=(0,0)$ and $p_n=(2^{-n},2^{-n})$ for $n\geq 1$. For any $n\geq 1$ choose a broken line segment $l_n$ without self-intersections, consisting of finitely many intervals, that starts at $p_0$, ends at $p_n$, has the total length of $2^n$, and is contained in $[[0,2^{-n})\times(2^{-n}-2^{-n-2},2^{-n}+2^{-n-2})]\cup\{p_n\}$. Define $P=\bigcup_{i=1}^\infty l_i$.
\end{defn}

\begin{figure}\centering\includegraphics[width=300px]{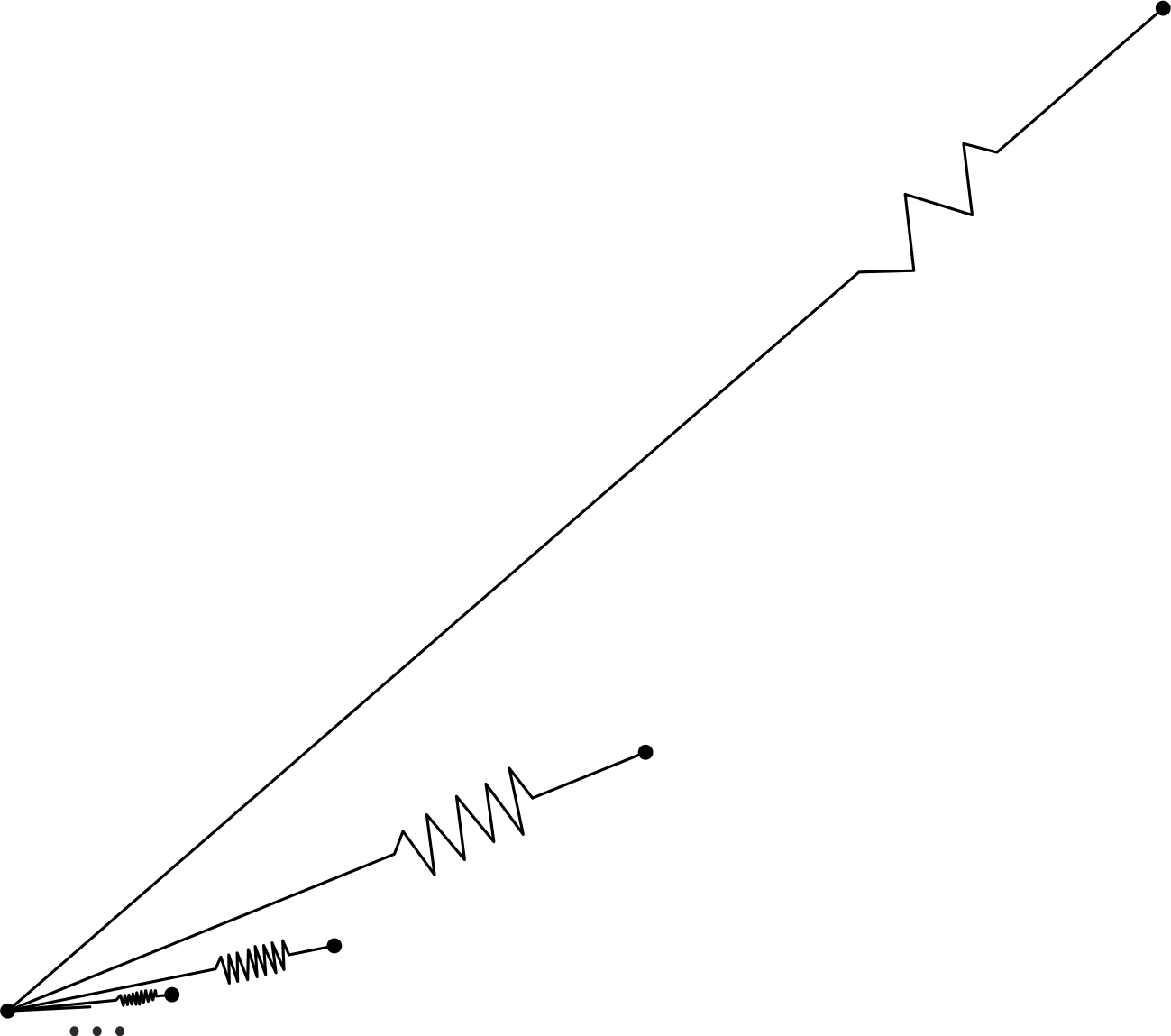}\caption{The space P}\end{figure}

\begin{thm}
The space $P$ is not an attractor of any weakly contracting IFS.
\end{thm}

\begin{proof}
Suppose that $f:P\ra P$ is a weak contraction. We shall examine how many of the points $p_i$ can belong to $f(P)$.

If $f(p_0)\neq p_0$ then there is a neighbourhood $U$ of $f(p_0)$ such that $d(p_0, U)>0$ and $U$ contains finitely many points $p_i$ and almost all of the sets $f(l_i)$. Note that only finitely many of the sets $f(l_i)$ may reach outside of $U$. Also observe that each $f(l_i)$ covers at most finitely many points $p_i$ because the lengths of $l_i$ are not increased by $f$ (this elementary property of contractions can be proven either by using $\delta$-chains, or, as in \cite{Kw}, by using the fact that $f$ does not increase one-dimensional measure). Consequently, only finitely many of the points $p_i$ belong to $f(P)$.

If, on the other hand, $f(p_0)=p_0$, then, given $n\geq 1$, note that $p_n$ may not belong to $f(l_i)$ for $i<n$, because the lengths of these sets are too small to traverse the whole $l_n$. But no other point in $P$ can be mapped onto $p_n$ by $f$, because $f$ decreases the distance between $p_0$ and any other point. Therefore, the only point $p_i$ present in $f(P)$ is $p_0$.

In conclusion, if $F$ is a weakly contracting IFS, then only finitely many of the points $p_i$ can belong to $F(P)$, and therefore $P$ is not an attractor of $F$.
\end{proof}

\section{acknowledgements}
The second author was supported by the ESF Human Capital Operational Programme grant 6/1/8.2.1/POKL/2009.

\end{document}